\documentclass{article}

\usepackage{amsmath}
\usepackage{amsthm}
\usepackage{latexsym}
\usepackage{amssymb}

\theoremstyle{definition}
\newtheorem{thm}{Theorem}[section]
\newtheorem{defi}[thm]{Definition}
\newtheorem{prop}[thm]{Proposition}
\newtheorem{cor}[thm]{Corollary}
\newtheorem{lemma}[thm]{Lemma}
\newtheorem{rem}[thm]{Remark}

\newcommand\ZZ{\ensuremath{\mathbb{Z}}}
\newcommand\CC{\ensuremath{\mathbb{C}}}

\newcommand\RR{\ensuremath{\mathbb{R}}}

\newcommand\QQ{\ensuremath{\mathbb{Q}}}

\newcommand\vect[1]{\mbox{\boldmath$#1$}}

\newcommand\ord{\mathrm{ord}}
\newcommand\val{\mathrm{val}}
\newcommand\Val{\mathrm{Val}}

\newcommand\kakko[1]{\left\vert\left\vert {#1} \right\vert\right\vert}

\renewcommand\tilde{\widetilde}

\author{Shinsuke Iwao\thanks{Graduate School of Science, Rikkyo University, 3-34-1 Nishi-Ikebukuro, Toshima-ku, Tokyo, 171-8501 Japan. E-mail: iwao@rikkyo.ac.jp} , 
Hidetomo Nagai\thanks{Faculty of Science and Engineering, Waseda University, 3-4-1, Okubo, Shinjuku-ku, Tokyo 169-8555, Japan. E-mail: hdnagai@aoni.waseda.jp}, 
and 
Shin Isojima\thanks{Department of Physics and Mathematics, Aoyama Gakuin University, 5-10-1 Fuchinobe, Chuo-ku, Sagamihara-shi, Kanagawa, 252-5258 Japan. E-mail: isojima@gem.aoyama.ac.jp}}
\title{Tropical Krichever construction for the non-periodic box and ball system}
\date{\today}

\begin{document}

\maketitle

\begin{abstract}
A solution for an initial value problem of the box and ball system is constructed from a solution of the periodic box and ball system. The construction is done through a specific limiting process based on the theory of tropical geometry. This method gives a tropical analogue of the Krichever construction, which is an algebro-geometric method to construct exact solutions to integrable systems, for the non-periodic system. 
\end{abstract}

\section{Introduction}
The box and ball system (BBS) is a cellular automaton expressed by infinite array of `.' (empty box with capacity $1$) and `$1$' (ball) \cite{TS, T}. If we put $U_n^t :=$ \{the number of ball in the $n$th box at time $t$ \} $\in \{0,1\}$, the evolution equation of the BBS is given by
\begin{equation*}
U_n^{t+1}=\min \left[1-U_n^t,\sum_{k=-\infty}^{n-1}{U_k^t}-\sum_{k=-\infty}^{n-1}{U_k^{t+1}} \right],
\end{equation*}
where we suppose that the number of `1' is finite. It is proved that any state of the BBS consists of only solitons. The BBS provides a simple illustration for the soliton interaction:
\begin{verbatim}
t=0:   ..111...11...1........
t=1:   .....111..11..1.......
t=2:   ........11..11.11.....
t=3:   ..........11..1..111..
\end{verbatim}
Since the BBS has the infinite number of conserved quantities, we may regard it as an integrable cellular automaton. Some integrable extensions of the BBS are proposed, that is, the BBS with arbitrary capacity of each box and/or species of the ball \cite{T,TTM} and/or with a carrier \cite{MINS,TM}.

In \cite{TTMS}, Tokihiro et al.~clarified that the BBS is directly obtained from the discrete Korteweg--de Vries (dKdV) equation through the technique of ultradiscretization. The exact solution to the BBS, 
\begin{equation*}
U_n^t=T_n^t+T_{n+1}^{t+1}-T_{n+1}^t-T_{n}^{t+1},\quad T_n^t=\min_{\lambda}\{a_\lambda n+b_\lambda t+c_\lambda\}
\end{equation*}
was also given by ultradiscretizing the $N$-soliton solution of the dKdV equation. On the other hand, the BBS can be constructed by applying the crystallization technique to a quantum integrable system. Many results from this context are also reported (see, for example, \cite{HHIKTT}).

We know some approaches to solve the initial value problem of the BBS. The method in \cite{Tak04} is based on the application of representation theory. A combinatorial method employing the theory of soliton equations is developed in \cite{Mada2}. Recently, an analogue of the dressing method is also reported \cite{WNSTG}.

Another extension of the BBS imposing the periodic boundary condition \cite{YT}
\begin{equation*}
U_{n+L}^t = U_n^t 
\end{equation*}
is actively studied. It is called the periodic box and ball system (pBBS). We refer to the period $L$ as \textit{system-size}. Its exact solution is given in terms of the tropical theta function (see appendix) as
\begin{equation*}
U_n^t=\Theta_n^t+\Theta_{n+1}^{t+1}-\Theta_{n+1}^t-\Theta_{n}^{t+1}.
\end{equation*}
A solution of the initial value problem of the pBBS is also studied from various contexts. As is the case in the BBS, the method using representation theory \cite{KTT} and the combinatorial method \cite{Mada} are known. An interesting method extending the solution for open boundary system to that for the periodic boundary system is reported in \cite{Mada3}. In addition, an analytical method employing the theory of tropical geometry shows remarkable development \cite{Inoue2,Iwao,Iwao2,Iwao3}. Since each method is based on its own theory, we may consider that the BBS and pBBS give a junction of various mathematical theories.

In this article, we propose a method to construct the solution of the BBS from that of the pBBS through a specific limiting process. We start with the results for the pBBS based on the theory of tropical curves (\S \ref{sec2}). In \S \ref{sec3}, we study system-size dependence of its solution. Then evaluating the limit as the system-size tends to infinity, we show that the solution of the pBBS reduces to that of the BBS. This process exactly gives a tropical analogue of the Krichever construction, which is an algebro-geometric method to construct exact solutions to integrable systems \cite{K1,K2,K3}. Finally, we give concluding remarks in \S 4.
\section{Review of the periodic box and ball system}\label{sec2}
In this section, we review the theory of the pBBS (See, for example, \cite{Inoue2,Iwao,Iwao3} for details).
\subsection{Spectral curve}\label{sec2.1}
Let $K=\bigcup_{d\in\ZZ_{>0}}{\CC((q^{1/d}))}$ be the Puiseux series field of indeterminate $q$. Let 
\begin{equation*}
\val:K\to \QQ\cup \{+\infty\}
\end{equation*}
be the valuation of $K$, where $\val(0)=+\infty$. The sub-ring $R=\{x\in K\,\vert\,\val(x)\geq 0\}$ is called the valuation ring. Then, $\val(x)=\max\{r\in\QQ\,\vert\,q^{-r}x\in R\}$ holds. 

We define $2\times 2$ matrices $T(y)$ and $H(y)$ by
\begin{equation*}
T(y)=\left(\begin{array}{@{\,}c@{\ }c@{\,}}
	q & 1 \\
	y & 1
\end{array}\right),\quad
H(y)=\left(\begin{array}{@{\,}c@{\ }c@{\,}}
	1 & 1 \\
	y & q
\end{array}\right)\quad
\in \mathrm{Mat}(2,R[y]).
\end{equation*}
For a given initial state of the pBBS, we define $\mathcal{X}(y) \in \mathrm{Mat}(2;R[y])$ as follows:
\begin{equation}\label{eq4}
.11...1...\quad \Rightarrow\quad \mathcal{X}(y):=HTTHHHTHHH.
\end{equation}
In other words, we replace `$.$' and `$1$' with $H(y)$ and $T(y)$, respectively, and define $\mathcal{X}(y)$ as their product.

We regard the characteristic polynomial of $\mathcal{X}(y)$, say $\Phi(x,y):=\det(\mathcal{X}(y)-xE)$ as a polynomial in $x$ and $y$ over $K$. The algebraic curve defined by $\Phi(x,y) = 0$ is called the spectral curve.
\subsection{Tropical theta function solutions}
In this subsection, we introduce tropical theta function solutions to the pBBS \cite{Iwao3}. See appendix \ref{appA} for fundamental results of the tropical geometry.
\subsubsection{Spectral curve and theta function solutions}
Let $\Phi=\Phi(x,y)$ be the characteristic polynomial of $\mathcal{X}$, and $\Gamma$ be the tropical curve defined by $\Phi$. Denote by $\iota:\Gamma\to\Gamma^0$ the surjection defined in \ref{a1}. Note that $\Gamma^0$ is a piecewise-linear subset of $\RR^2$. The tropical curve $\Gamma$ is given by means of the following propositions:
\begin{prop}[\cite{Inoue2}]\label{prop2.2}
For a given initial state of pBBS with period $L$, let $S_1\leq S_2\leq \cdots \leq S_g$ be the lengths of solitons. Then, $\Gamma^0$ is illustrated as Figure \ref{fig1}, where $A_i:=\sum_{k=1}^g{\min[S_i,S_k]}$.
\end{prop}

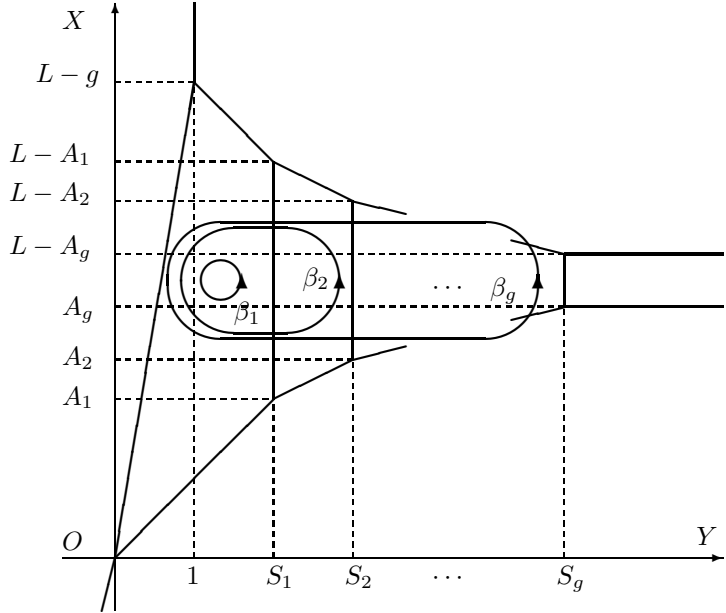
\begin{figure}[htbp]
\begin{center}
\begin{picture}(250,240)
\thinlines
\put(0,20){\vector(1,0){250}}
\put(20,0){\vector(0,1){230}}
\thicklines
\put(20,20){\line(1,1){60}}
\put(80,80){\line(2,1){30}}
\put(110,95){\line(4,1){20}}
\put(190,135){\line(1,0){60}}
\put(190,135){\line(-4,1){20}}
\put(20,20){\line(1,6){30}}
\put(50,200){\line(1,-1){30}}
\put(80,170){\line(2,-1){30}}
\put(110,155){\line(4,-1){20}}
\put(190,115){\line(1,0){60}}
\put(190,115){\line(-4,-1){20}}
\put(80,80){\line(0,1){90}}
\put(110,95){\line(0,1){60}}
\put(140,120){$\cdots$}
\put(190,115){\line(0,1){20}}
\put(50,200){\line(0,1){30}}
\put(20,20){\line(-1,-4){5}}
\thinlines
\multiput(50,20)(0,4){45}{\line(0,1){2}}
\multiput(80,20)(0,4){15}{\line(0,1){2}}
\multiput(110,20)(0,4){20}{\line(0,1){2}}
\multiput(190,20)(0,4){25}{\line(0,1){2}}
\multiput(20,200)(4,0){8}{\line(1,0){2}}
\multiput(20,170)(4,0){15}{\line(1,0){2}}
\multiput(20,155)(4,0){23}{\line(1,0){2}}
\multiput(20,135)(4,0){43}{\line(1,0){2}}
\put(-10,200){$L-g$}
\put(-20,170){$L-A_1$}
\put(-20,155){$L-A_2$}
\put(-20,135){$L-A_g$}
\multiput(20,80)(4,0){15}{\line(1,0){2}}
\multiput(20,95)(4,0){23}{\line(1,0){2}}
\multiput(20,115)(4,0){43}{\line(1,0){2}}
\put(0,78){$A_1$}
\put(0,93){$A_2$}
\put(0,112){$A_g$}
\put(47,10){$1$}
\put(77,10){$S_1$}
\put(107,10){$S_2$}
\put(140,10){$\cdots$}
\put(187,10){$S_g$}
\put(240,25){$Y$}
\put(0,220){$X$}
\put(0,23){$O$}
\thicklines
\put(60,125){\circle{15}}
\put(68,128){\vector(0,1){0}}
\put(75,125){\oval(60,40)}
\put(105,128){\vector(0,1){0}}
\put(110,125){\oval(140,44)}
\put(180,128){\vector(0,1){0}}
\put(65,110){$\beta_1$}
\put(91,123){$\beta_2$}
\put(162,119){$\beta_g$}
\end{picture}
\end{center}
\caption{Example of $\Gamma^0$.}
\label{fig1}
\end{figure}
Define the segments $\gamma^+, \gamma^-:[0,S_g]\to\Gamma^0$ by
\begin{equation*}
\gamma^+(t)=(\sum_i\min[S_i,t],t),\quad
\gamma^-(t)=(\min{[(L-g)t,L-\sum_i\min{[S_i,t]}]},t)
\end{equation*}
and $\theta_i:(0,L-2A_i)\to\Gamma^0$ ($i=1,2,\dots,g$) by $ \theta_i(t)=(t+A_i,S_i) $. Moreover, we define the half-lines $\sigma_1,\dots,\sigma_4:(0,+\infty)\to\Gamma^0$ as
\begin{equation*}
\sigma_1(t)=(-Lt,-2t),\ \sigma_2(t)=(t+L-g,1),\ \sigma_3(t)=(A_g,t+S_g),\ 
\sigma_4(t)=(L-A_g,t+S_g).
\end{equation*}
Then, it follows that
$\Gamma^0=(\gamma^+\cup\gamma^-)\cup(\bigcup_{i=1}^g{\theta_i})\cup (\bigcup_{i=1}^4{\sigma_i})$. Note that $\theta_i$ corresponds with $\theta_{i+1}$ if there exists $i$ such that $S_i=S_{i+1}$.

\begin{prop}
Let $\Gamma$ be the tropical curve and $\iota:\Gamma\to\Gamma^0$ be the above surjection. Then, $\Gamma$ is the unique connected finite graph such that (i) $\Gamma=(\gamma^+\cup\gamma^-)\cup(\coprod_{i=1}^g{\theta_i})\cup (\bigcup_{i=1}^4{\sigma_i})$, (ii) $\iota(\gamma^{\pm})=\gamma^\pm$, $\iota(\theta_i)=\theta_i$, $\iota(\sigma_i)=\sigma_i$. 
In other words, we obtain $\Gamma$ from $\Gamma^0$ by distinguishing $\theta_i$ and $\theta_j$ for $i\neq j$.
\end{prop}

Let $\beta_i \in H_1(\Gamma,\ZZ)$ $(i=1,\dots,g)$ be $g$ cycles over $\Gamma$ (see Figure \ref{fig1}) expressed as
\begin{equation*}
\beta_i(t)=\left\{
\begin{array}{lc}
\gamma^+(t) & 0\leq t\leq S_i \\
\theta_i(t-S_i) & S_i<t<L-2A_i+S_i\\
\gamma^-(-t+L-2A_i+2S_i) & L-2A_i+S_i\leq t\leq L-2A_i+2S_i.
\end{array}
\right.
\end{equation*}
Note that $\beta_1,\dots,\beta_g$ give a basis of $H_1(\Gamma,\ZZ)$. For the basis $\{\beta_i\}_i$, the period matrix $B$ and the Jacobi variety $\mathrm{Jac}\ \Gamma:=\RR^g/B\ZZ^g$ are determined (see definition \ref{defA5}). 
\begin{cor}\label{crl2.3}
We have $B=((L-2A_i)\delta_{i,j}+2\min[S_i,S_j])_{i,j}$.
\end{cor}
For a sufficiently large real number $U>0$, we put 
\begin{equation*}
O:(X,Y)=(0,0),\quad Q:(X,Y)=(U,1),\quad R:(X,Y)=(A_g,U).
\end{equation*}
Furthermore, we define real vectors $\mu,\omega\in \mathrm{Jac}\,\Gamma$ by
\begin{equation}
\mu:
=-\vect{A}(Q),\quad 
\omega:=
-\vect{A}(R),
\end{equation}
where $\vect{A}:\mathrm{Pic}(C)\to \mathrm{Jac}\,\Gamma$ is the tropical Abel-Jacobi mapping with initial point $O$. We note that $\mu$ and $\omega$ are independent of $U$. From proposition \ref{prop2.2}, we have
\begin{equation}\label{eq5}
\mu\equiv {}^t(1,1,\dots,1),\quad \omega\equiv
-{}^t(S_1,S_2,\dots,S_g),\quad(\mathrm{mod\,}B\ZZ^g).
\end{equation}

\begin{rem}
Using the expression $B=(\vect{b}_1,\dots,\vect{b}_g)$, we have $L \mu={}^t(L,L,\dots,L) = \vect{b}_1+\vect{b}_2+\cdots+\vect{b}_g \equiv 0\ (\mathrm{mod\,}B\ZZ^g)$, which reflects the periodic boundary condition $U_{n+L}^t=U_n^t$. See theorem \ref{thm2.4}.
\end{rem}

The general solution to the pBBS is represented by the tropical theta function associated with the tropical curve. Let $\Theta(z;B)$ be the tropical theta function defined in appendix \ref{appA}. 
\begin{thm}[\cite{Iwao2}]\label{thm2.4}
The general solution to the pBBS is expressed as
\begin{equation*}
U_n^t=\Theta_n^t+\Theta_{n+1}^{t+1}-\Theta_n^{t+1}-\Theta_{n+1}^t,\qquad \Theta_n^t=\Theta(\mu n+\omega t+c_0;B),
\end{equation*}
where $c_0\in \mathrm{Jac}\,\Gamma$ is determined from an initial state through the eigenvector mapping (See the following).
\end{thm}
\subsubsection{Eigenvector mapping}

Here we review the eigenvector mapping \cite{Mumford2}, which maps an initial data of a discrete integrable system to a positive divisor on some algebraic curve. We here consider only a special case necessary for our purpose.

Let $\mathcal{X}=\mathcal{X}(y)$ be the $2\times 2$ matrix defined in \S \ref{sec2.1} and $C$ be the algebraic curve defined from the algebraic equation $f(x,y)=\det(\mathcal{X}(y)-xE)=0$. Denote the genus of $C$ by $\tilde{g}$. It is known that the genus of the tropical curve $\Gamma$ associated with $f$ is equal or less than $\tilde{g}$.

\begin{lemma}[\cite{Mumford2}]
There exists an unique point $\infty\in C$ which is expressed as $(x,y)=(\infty,\infty)$.
\end{lemma}

Due to the definition of $C$, the equation
\begin{equation}\label{eq6}
\mathcal{X}(y) {}^t(1,h(x,y))=x {}^t(1,h(x,y))
\end{equation}
determines the rational function $h$ on $C$ uniquely.

\begin{thm}[\cite{Mumford2}]\label{thm2.7}
Let $\mathrm{Div}\,C:=\bigoplus_{p\in P}{\ZZ\cdot p}$ be the divisor group of $C$, which is a free abelian group generated by the points in $C$. Then, there exists a general positive divisor $D\in\mathrm{Div}\,C$ of degree $\tilde{g}$ such that the zero divisor $(h)_0$ of $h$ is expressed as
\begin{equation}\label{eq7}
(h)_0=D+\infty.
\end{equation}
\end{thm}

In general, the mapping $\mathcal{X}(y)\mapsto D$ is called the eigenvector mapping. Denote $D=p_1+p_2+\dots+p_{\tilde{g}}$ $(p_k\in C)$ and the $(x,y)$-coordinate of the point $p_k$ by $(x_k,y_k)$. Let $X_k:=\val(x_k),Y_k:=\val(y_k)$. We regard the point $P_k:(X,Y)=(X_k,Y_k)$ in the real plane $\RR^2$ as the tropicalization of the point $p_k$. It is immediately proved that $P_k\in\Gamma$. (See, for example, \cite{Iwao3}).

The element $c_0\in\mathrm{Jac}\,\Gamma$ in theorem \ref{thm2.4} is calculated from the configuration of $P_k$.
\begin{prop}[\cite{Iwao2}]\label{prop2.9}
Let $\kappa\in\mathrm{Jac}\,\Gamma$ be the Riemann constant of $\Gamma$ (\S \ref{sec:a.1.5}). Then, we have
\begin{equation}\label{c0}
c_0=
\kappa-\vect{A}(P_1+\cdots+P_{\tilde{g}}). 
\end{equation}
\end{prop}
\subsubsection{Riemann constant}

For general tropical curves, it is hard to calculate the Riemann constant $\kappa$ explicitly. However, for our tropical curves, we have the following simple expression.
\begin{prop}
\[
\kappa\equiv \frac{L}{2} \cdot {}^t(1,1,\dots,1).
\]
\end{prop}

To prove the proposition, we need some ideas on the tropical geometry introduced in the appendix. Let $\tilde{\vect{A}}:\Gamma\to\RR^g$ be the multi-valued function defined in \S \ref{sec:a.1.5}. Define the multi-valued function $f:\Gamma\to \RR$ by $f(P):=\Theta(\tilde{\vect{A}}(P);B)$.

Denote the point $\theta_i(L/2-A_i)$ in $\Gamma$ by $Q_i$. The set $U=\Gamma\setminus\{Q_1,\dots,Q_g\}$ is a simply connected open subset of $\Gamma$. Let $\vect{A}_0:U\to\RR^g$ be the restriction of $\tilde{\vect{A}}$ to $U$, where we take the branch as $\vect{A}_0(O)=0$. Then $\vect{A}_0$ is single-valued. Define a single-valued function $f_0:U\to \RR$ by $f_0(P):=\Theta(\vect{A}_0(P);B)$.

\begin{lemma}
The point $Q_i$ ($i=1,2,\dots,g$) is a zero of $f$ of degree $1$.
\end{lemma}
\begin{proof}
We may take the branch of $f$ such that $0<t<L/2-A_i\ \Rightarrow\ f(\theta_i(t))=f_0(\theta_i(t))$ without loss of generality. The value of tropical Abel-Jacobi mapping is calculated as $\vect{A}(\theta_i(t))=(t\delta_{i,j}+\min[S_i,S_j])_{j=1}^g$. From the equation
\[
(t\delta_{i,j}+\min[S_i,S_j])_{j=1}^g=\frac{1}{2}Be_i+(t-L/2+A_i)e_i,
\]
we have $f(\theta_i(t))=\Theta(\frac{1}{2}Be_i+(t-L/2+A_i)e_i;B)$. By lemma \ref{half}, we obtain the result.
\end{proof}

\textbf{Proof of the proposition.} By lemma \ref{ZERO}, the set of zeros of $f$ is $\{Q_1,\dots,Q_g\}$. It gives the result.$\qed$
\subsection{Example}
We consider the pBBS of system-size $L = 10$ and the initial state given in \eqref{eq4}. We have $\tilde{g}=2$, $S_1 = 1$, $S_2 = 2$, and therefore obtain $A_1 = 2$, $A_2 = 3$ and
\begin{equation*}
B = 
\begin{pmatrix}
8 & 2 \\
2 & 8
\end{pmatrix}.
\end{equation*}
By calculating $h(x,y)$ in \eqref{eq6} associated with $\mathcal{X}(y)$ in \eqref{eq4}, one can find $P_1 : (2,1)$ and $P_2 : (5,1)$. Their images by the eigenvector mapping are ${}^t (1,1)$ and ${}^t (4,1)$, respectively. Hence, we have $c_0 = {}^t (0,3)$. Now, we obtain the tropical theta function
\begin{equation}\label{eq8}
\Theta_n^t = \Theta \left(  
n \begin{pmatrix} 1 \\ 1 \end{pmatrix} -
t \begin{pmatrix} 1 \\ 2 \end{pmatrix} +
\begin{pmatrix} 0 \\ 3 \end{pmatrix};
\begin{pmatrix}
8 & 2 \\
2 & 8
\end{pmatrix}
\right),
\end{equation}
which expresses the solution for this initial state through theorem \ref{thm2.4}.
\section{Solutions to the non-periodic system}\label{sec3}

In this section, we propose the method to construct solutions to the BBS from the general solutions to the pBBS. Let $\mathcal{X} = \mathcal{X}(y)$ be the $2 \times 2$ matrix defined in the previous section. In the same way, we denote the tropical curve associated with the pBBS by $\Gamma$, and its period matrix by $B$. Moreover, denote $\mathcal{X}[M] := \mathcal{X}H^M (M \ge 1)$. We are interested in the behavior of solutions to the periodic system when $M \to +\infty$. Especially, it implies system-size $L \to +\infty$. Then let us consider the behavior of the function $\Theta^t_n = \Theta (\mu n + \omega t + c_0;B)$ when $L \to +\infty$.

Let $S=-2\,\mathrm{diag}(A_1,\dots,A_g)+(2\min[S_i,S_j])_{i,j}$. Then, we have $B=LE+S$ from corollary \ref{crl2.3}. Define $z=\mu n+\omega t+c_0$. From (\ref{eq5}), the vectors $\mu$ and $\omega$ do not depend on $L$. Due to the equation \eqref{c0}, the problem is reduced to considering the asymptotic behavior of the eigenvector mapping $\mathcal{X}(y)\mapsto D=P_1+\dots+P_{\tilde{g}}$ (theorem \ref{thm2.7}) when $L\to +\infty$.

The value of the eigenvector mapping is characterized by the relations (\ref{eq6}) and (\ref{eq7}). More precisely, we have the following:
\begin{lemma}\label{root0}
Denote the $(i,j)$-component of $\mathcal{X}(y)$ by $a_{i,j}(y)\in K[y]$. The point $(X,Y)\in\Gamma$ is contained in $\{P_1,\dots,P_{\tilde{g}}\}$ if and only if : (i) the polynomial $a_{2,1}(y)$ has a root $y^\ast\in K$ such that $\val(y^\ast)=Y$, (ii) $X=\val(a_{1,1}(y^\ast))$.
\end{lemma}
\begin{proof}
By (\ref{eq6}) and (\ref{eq7}), we have: $(X,Y)\in\Gamma$ is contained in $\{P_1,\dots,P_{\tilde{g}}\}$ $\iff$ there exist two elements $x^\ast,y^\ast\in K$ such that $X=\val(x^\ast)$, $Y=\val(y^\ast)$ and $\mathcal{X}(y^\ast)\cdot {}^t(1,0)=x^\ast\cdot {}^t(1,0)$. This implies the lemma.
\end{proof}
Furthermore, denote by $\Gamma[M]$ the tropical curve defined from $\mathcal{X}[M](y)$, and by $P_1[M], \dots, P_{\tilde{g}}[M] \in \Gamma[M]$ the image of $\mathcal{X}[M](y)$. We have the following theorem, of which proof is given in the next subsection.
\begin{thm}\label{thm3.7}
There exists some large number $m_0>0$ such that for all $M>m_0$, the value $\vect{A}_0(P_1[M]+\dots+P_{\tilde{g}}[M])\in\RR^g$ is constant.
\end{thm}
\subsection{Proof of theorem \ref{thm3.7}}
First, we introduce some notations for the proof of the theorem. Let $R:=\{x\in K\,\vert\,\val(x)\geq 0\}$ be the valuation ring of $K$. Define two subsets $K^+, R^+$ of $K$ by $K^+:=\{c_0q^{n/d}+c_1q^{(n+1)/d}+c_2q^{(n+2)/d}+\cdots\in K\,\vert\,c_0>0\}$ and $R^+:=R\cap K^+$. These four sets naturally have an algebraic structure as below:
\[
K: \mbox{field},\quad
R: \mbox{ring},\quad
K^+: \mbox{semi-field},\quad
R^+: \mbox{semi-ring}.
\]
The semi-ring $R^+$ will not be used in the sequel.

Let $\mathcal{O}$ be a $K$-algebra. $\mathcal{O}$ is also regarded as an $R$-algebra. For an $R$-subalgebra $\mathcal{O}_R\subset \mathcal{O}$, we define the mapping $v(\ \cdot\ ;\mathcal{O}_R):\mathcal{O}\to\RR\cup\{+\infty\}$ by $v(x;\mathcal{O}_R):=\sup\{r\,\vert\,q^{-r}x\in\mathcal{O}_R\}$. For example, if $\mathcal{O}=K$ and $\mathcal{O}_R=R$, then $v(x;R)=\val(x)$.

We fix $\mathcal{O}=K[y]$ in the rest of the paper. For a rational number $p$, we denote by $\mathcal{O}_R(p)$ the $R$-subalgebra  of $K[y]$ generated by $q^{-p}y$. Let $v_p(\, \cdot\, ):=v(\ \cdot\ ;\mathcal{O}_R(p))$. It immediately follows that $v_p(q^my^n)=m+np$ for any rational number $m$ and any integer $n$.

For any two elements $f,g$ of $K[y]$, we have (i) $v_p(fg)=v_p(f)+v_p(g)$, (ii) $v_p(f+g)\leq \min[v_p(f),v_p(g)]$. Let $K^+[y]$ be the $K^+$-subsemifield of $K[y]$ generated by $y$. By restricting $v_p$ on $K^+[y]$, the above inequality (ii) is improved as follows:
\begin{equation}
f,g\in K^+[y]\ \Rightarrow\ v_p(f+g)=\min[v_p(f),v_p(g)].
\end{equation}

Because $\mathcal{O}_R(p)$ is an $R$-algebra, the substitution $q\mapsto 0$ is well-defined over $\mathcal{O}_R(p)$. For example, for the element $y\in \mathcal{O}_R(p)$ ($p\geq 0$), the substitution procedure is expressed as
\begin{equation*}
y=q^p(q^{-p}y)\mapsto 
\left\{
\begin{array}{ll}
y, & p=0, \\
0, & p>0.
\end{array}
\right.
\end{equation*}
For $f \in \mathcal{O}_R(p)$, we denote by $f_{(p)}$ the element of $\CC[q^{-p}y]$ which is obtained by substitution $q\mapsto 0$. For example, $y_{(0)}=y$, $(q^{-1}y+2y^2+3q^{-2}y^2)_{(1)}=q^{-1}y+3q^{-2}y^2$, \textit{etc.}

\begin{rem}
Let $I=\{x\in R\,\vert\,\val(x)>0\}$ be the maximal ideal of $R$. Then, the substitution $q\mapsto 0$ is equivalent to the algebraic mapping $\mathcal{O}_R(p)\to \mathcal{O}_R(p)\otimes_R(R/I)$; $f\mapsto f\otimes 1$. In fact, under the natural isomorphism $\mathcal{O}_R(p)\otimes_R(R/I)\simeq\CC[q^{-p}y]$, we have $f\otimes 1\simeq f_{(p)}$. It would be more natural to express the substitution procedure by the algebraic mapping. However, we do not use this expression to avoid redundant notations hereafter.
\end{rem}

\begin{rem}\label{rem1}
By the definition of $v_p$, we have $q^{-v_p(f)} \cdot f\in\mathcal{O}_R(p)$ for any non-zero $f\in K[y] $. We can check $(q^{-v_p(f)}\cdot f)_{(p)}\neq 0$, which is not always true for arbitrary $K$-algebra $\mathcal{O}$ and $R$-subalgebra $\mathcal{O}_R$. In fact, this is not true for $\mathcal{O}=K((y))$, $\mathcal{O}_R=R((y))$.
\end{rem}

\begin{lemma}\label{lem3.4}
For non-zero $f\in K[y]$, there uniquely exists a non-zero polynomial $\mathrm{in}_p(f)\in\CC[q^{-p}y]$ such that $v_p(f-q^{v_p(f)}\cdot \mathrm{in}_p(f))>v_p(f)$.
\end{lemma}
\begin{proof}
It is sufficient to put $\mathrm{in}_p(f):=(q^{-v_p(f)}\cdot f)_{(p)}$. In fact, we have $ (q^{-v_p(f)}\cdot f- \mathrm{in}_p(f))_{(p)}=0$, which implies $v_p(q^{-v_p(f)}\cdot f- \mathrm{in}_p(f))>v_p(q^{-v_p(f)}\cdot f)$.
\end{proof}

\begin{rem}
In the situation as in lemma \ref{lem3.4}, we denote $f=q^{v_p(f)}\cdot\mathrm{in}_p(f)+o_p(q^{v_p(f)})$.
\end{rem}
\begin{lemma}\label{root}
For $f\in K[y]$, the following (i) and (ii) are equivalent:
\begin{enumerate}
\def\labelenumi{(\theenumi)}
\def\theenumi{\roman{enumi}}
\item There exists some $y^\ast\in K$ such that $f(y^\ast)=0$ and $\val(y^\ast)=p$.
\item The polynomial $\mathrm{in}_p(f)$ has a non-zero root.
\end{enumerate}
\end{lemma}

Next, we prove the following lemma. 
\begin{lemma}
\begin{enumerate}
\def\labelenumi{(\theenumi)}
\def\theenumi{\roman{enumi}}
\item There exists some $y_0$ such that $\mathrm{in}_Y(a_{i,1}[M])$ ($i=1,2, M=0,1,2,\dots$) is a monomial for all $Y>y_0$,
\item If $M>Y>0$, then $\mathrm{in}_Y(a_{i,j}[M])=\mathrm{in}_Y(a_{i,j}[Y+1])$,
\item A point $P_i[M]\in\Gamma$ which is contained in $\{(X,Y) | Y<0\}$ satisfies $\vect{A}_0(P_i[M])=0$.
\end{enumerate}
\end{lemma}
\begin{proof}
Proof of (i): Note that for any $f=y^m(a_0+a_1y+\dots+a_ny^n)\in K[y]$, ($a_i\in K,a_0\neq 0$), the inequality $p>\max_i[\frac{1}{i}(\val({a_0}) - \val(a_i))]$ implies $\mathrm{in}_p(f)=\mathrm{in}_p(a_0y^m)$. Because of the expression $H,T=\{\mbox{an upper triangle matrix over $K^+$}\}+y\cdot\{\mbox{an lower triangle matrix over $K^+$}\}$, we have the following relations for sufficiently large $\zeta$.
\[
Y>\zeta \quad\Rightarrow\quad
\begin{array}{c}
\mathrm{in}_Y(a_{1,1})=\mathrm{in}_Y(\alpha),\quad
\mathrm{in}_Y(a_{1,2})=\mathrm{in}_Y(\beta),\\
\mathrm{in}_Y(a_{2,1})=\mathrm{in}_Y(\gamma y),\quad
\mathrm{in}_Y(a_{2,2})=\mathrm{in}_Y(\delta+\epsilon y),
\end{array}
\]
where $\alpha,\beta,\gamma,\delta,\epsilon\in K^+$. 

Define inductively $\alpha[M],\beta[M],\dots\in K^+$ $(M=0,1,2,\dots)$ by the formulas $\alpha[M+1]=\alpha[M]$, $\beta[M+1]=\alpha[M]+q\beta[M]$, $\gamma[M+1]=\gamma[M]+\delta[M]$, $\delta[M+1]=q\delta[M]$, $\epsilon[M+1]=\gamma[M]+\epsilon[M]$, where $\alpha[0]=\alpha, \beta[0]=\beta,\dots$ \textit{etc.} Let $y_0:=\max[\zeta,\val(\beta)-\val(\alpha),\val(\gamma)-\val(\epsilon)]$. Then, we have the following relations by induction on $M$.
\[
Y>y_0\quad\Rightarrow\quad
\begin{array}{c}
\mathrm{in}_Y(a_{1,1}[M])=\mathrm{in}_Y(\alpha[M]),\quad
\mathrm{in}_Y(a_{1,2}[M])=\mathrm{in}_Y(\beta[M]),\\
\mathrm{in}_Y(a_{2,1}[M])=\mathrm{in}_Y(\gamma[M] y),\quad
\mathrm{in}_Y(a_{2,2}[M])=\mathrm{in}_Y(\delta[M]+\epsilon[M] y).
\end{array}
\]

Proof of (ii): Let $h_{i,j}[M]$ be the $(i,j)$-component of $H^M$. Then it follows that
\begin{equation}\label{in}
M>Y\ \Rightarrow\ 
\begin{array}{c}
h_{1,1}[M]=1+o_Y(1),\quad h_{1,2}[M]=1+o_Y(1),\\
h_{2,1}[M]=y+o_Y(q^Y),\quad h_{2,2}[M]=y+o_Y(q^Y).
\end{array}
\end{equation}
This implies $v_Y(a_{i,1}[M+1])=v_Y(a_{i,2}[M+1])=v_Y(a_{i,1}[M]+ya_{i,2}[M])$, ($i=1,2$). Again using (\ref{in}), we obtain $$\mathrm{in}_Y(a_{i,1}[M+2])=\mathrm{in}_Y(a_{i,1}[M+1]+ya_{i,2}[M+1])=\mathrm{in}_Y(a_{i,1}[M+1]),\ \ i=1,2.$$ Therefore, $M>Y\Rightarrow\mathrm{in}_Y(a_{i,1}[M])=\mathrm{in}_Y(a_{i,1}[Y+1])$.

Proof of (iii): It is straightforward from proposition \ref{prop2.2}.
\end{proof}
Now, we proceed for the proof of the theorem. 
\begin{proof}
It follows that any point $P_i[M]$ is contained in the domain $\{Y\leq y_0\}$ from (i) and lemma \ref{root0}. By (iii), we need not consider any points $P_i[M]$ contained in the domain $\{Y<0\}$. From (ii) and lemma \ref{root0}, the $X$ and $Y$-coordinates of points $P_i[M]$ contained in the domain $\{0\leq Y\leq y_0\}$ are constant for any $M$ greater than $y_0$. Then we obtain the theorem when we put $m_0:=y_0$.
\end{proof}

\subsection{Limiting procedure for the tropical theta function solution}

By the results of the previous sections, we have $$B=LE+S,\quad z=L/2\cdot{}^t(1,1,\dots,1)+\mu n+\omega t+c$$ for sufficiently large $L$. Substituting this to the tropical theta function, we get
\[
\Theta(z;B)=\min_{r\in\ZZ^g}[\frac{L}{2}(\langle r,r\rangle+\langle r,e\rangle)+
\frac{1}{2}\langle r,Sr\rangle+\langle r,\mu n+\omega t+c\rangle],
\]
where $e=e_1+\dots+e_g$. When $L$ tends to $+\infty$, it follows that
\begin{align*}
\Theta(z;B)
&=\min_{r\in\{-1,0\}^g}[\frac{L}{2}(\langle r,r\rangle+\langle r,e\rangle)+
\frac{1}{2}\langle r,Sr\rangle+\langle r,\mu n+\omega t+c\rangle]\\
&=\min_{r\in\{0,1\}^g}[\frac{L}{2}(\langle r,r\rangle-\langle r,e\rangle)+
\frac{1}{2}\langle r,Sr\rangle-\langle r,\mu n+\omega t+c\rangle]\\
&\stackrel{L\to+\infty}{\longrightarrow}\min_{r\in\{0,1\}^g}[\frac{1}{2}\langle r,Sr\rangle-\langle r,\mu n+\omega t+c\rangle].
\end{align*}

Let $T_n^t=\min_{r\in \{0,1\}^g}[\frac{1}{2}\langle r,Sr\rangle-\langle r,\mu n+\omega t+c\rangle]$. Then, we obtain the main theorem in this paper.
\begin{thm}
Let $U_n^t$ be the general solution of pBBS. The $\tilde{U}_n^t$ below satisfies the BBS.
\[
\tilde{U}_n^t:=
\lim_{L\to +\infty}{U_n^t}=T_n^t+T_{n+1}^{t+1}-T_{n+1}^t-T_n^{t+1}.
\]
\end{thm}

\begin{rem}
This limiting procedure is a tropical analog of the method in \cite{Mumford} Part III b \S 4.
\end{rem}
\subsection{Example}
For example, from the solution of the pBBS \eqref{eq8}, we obtain 
\begin{equation*}
T_n^t = \min [0, -n+t-3, -n+2t-1, -2n+3t-8 ],
\end{equation*}
which gives the two-soliton solution of the BBS.
\section{Concluding remarks}
In this paper, starting with a given initial state of the pBBS, we have constructed the solution of the BBS. Its initial state is given by inserting infinite number of empty boxes into that of the pBBS. Theorem \ref{thm3.7}, which evaluates the asymptotic behaviour of the eigenvector mapping, plays an essential role when the system-size tends to infinity. The obtained solution is identical to the well-known soliton solution of BBS. Moreover, our method gives a tropical analogue of the Krichever construction and establishes an application of tropical geometry to integrable cellular automata.

Our method can be readily extended to a system represented by Lax pair, for example, the two-dimensional box and ball system and the discrete Painlev\'{e} equations \cite{murata,tuda}. It is a future problem to study such systems in the context of the tropical geometry.
\section*{Acknowledge}
This work was supported by JSPS KAKENHI 21-1939, 23740091 and 21760063.
\appendix

\section{Review of the tropical geometry}\label{appA}

In this section, we introduce some ideas of the tropical geometry. For details, see \cite{IMS}.

\subsection{Definitions}

\subsubsection{tropical curve}

Let $K$ be the Puiseux series field of indeterminate $q$ over $\CC$, and $\val:K\to \QQ\cup\{+\infty\}$ be the valuation. For a polynomial \[
\textstyle
\Phi(x,y)=\sum_{w=(w_1,w_2)\in\ZZ^2}{a_wx^{w_1}y^{w_2}}\qquad
\mbox{($a_w=0$ except for finitely many $w$)}
\]
in $x$ and $y$ over $K$, we define $\textstyle \Val_\Phi(X,Y):=\min_w[\val(a_w)+w_1X+w_2Y]$. Let $\Gamma^0$ be the subset of $\RR^2$ defined by
\[
\Gamma^0=\{(A,B)\in\RR^2\,\vert\,\mbox{the continuous map }\Val_\Phi:\RR^2\to \RR\mbox{ is indifferentiable at $(X,Y)=(A,B)$}\}.
\]
For a point $P \in \Gamma^0$, define the finite set $\Lambda(P)\subset \ZZ^2$ by
\[
\Lambda(P):=\{w\in\ZZ^2\,\vert\,\Val_\Phi(X,Y)=\val(a_w)+w_1X+w_2Y\},\quad P=(X,Y).
\]
Let $\Phi_P$ be the polynomial over $K$ defined by $\Phi_P:=\sum_{w\in \Lambda(P)}{a_wx^{w_1}y^{w_2}}$.

We regard the polynomial $\Phi_P$ as an element of the extended ring $K[x,x^{-1},y,y^{-1}]$. The multiplicity of the point $P$ is a positive number $\vartheta(P)$ which is defined by the following formula.
\[
\vartheta(P):=\sharp \{\mbox{irreducible components of $\Phi_P\in K[x^{\pm 1},y^{\pm 1}]$}\}.
\]

\begin{defi}\label{a1}
The tropical curve associated with $\Phi$ is a connected finite graph, which may have edges of infinite length such that (i) there exists a finite surjection $\iota:\Gamma\to\Gamma^0$ which is locally isomorphic except for finitely many points, (ii) $\sharp\{\iota^{-1}(P)\}=\vartheta(P)$ for all $P\in\Gamma^0$. Such $\Gamma$ uniquely exists.
\end{defi}

Note that the slope of edges of $\Gamma^0$ is a rational number. By using the lattice length of $\RR^2$, we equip $\Gamma^0$ with the structure of metric graph. Pulling back by the almost locally isomorphic map $\iota$, we equip $\Gamma$ with metric as well. For two points $P$, $Q$ $\in$ $\Gamma$, denote by $\mathrm{dist}(P,Q)$ the distance between $P$ and $Q$. 

\subsubsection{Piecewise linear functions over $\Gamma$}

A continuous function $f:\Gamma\to\RR$ is piecewise linear if the limit 
\[
D_\gamma(f):=\lim_{t\to 0^+}{\frac{f(\gamma(t))-f(\gamma(0))}{\mathrm{dist}(\gamma(t),\gamma(0))}}
\]
is an integer for any path $\gamma:[0,1]\to \Gamma$. Let $\mathcal{L}$ be the set of piecewise linear functions over $\Gamma$. For a point $P\in\Gamma$, we define the subset $T_P\subset \mathrm{Hom}(\mathcal{L},\ZZ)$ by $T_P:=\{D_\gamma:\mathcal{L}\to\ZZ\,\vert\,\gamma(0)=P\}$. Note that the set $T_P$ is finite for any $P$.

The degree of $f$ at $P$ is defined by $\ord_P(f):=-\sum_{D_\gamma\in T_P}{D_{\gamma}f}$. By definition, we have $\ord_P(f+g)=\ord_P(f)+\ord_P(g)$, $\ord_P(kf)=k\cdot\ord_P(f)$ ($k\in\ZZ$).

\begin{defi}
A piecewise linear function $f:\Gamma\to\RR$ is rational if (i) $f$ is bounded (ii) $\ord_P(f)\neq 0$ for finitely many $P$. We say that $P$ is a zero of $f$ if $\ord_P(f)>0$, and that $P$ is a pole of $f$ if $\ord_P(f)<0$.
\end{defi}

\subsubsection{Divisor group and Picard group}

Let $\mathrm{Div}\,\Gamma:=\bigoplus_{P\in\Gamma}{\ZZ\cdot P}$ be the free abelian group generated by points in $\Gamma$. For a rational function $f$, define the divisor $(f)$ by $(f):=\sum_{P\in\Gamma}{\mathrm{ord}_P(f)\cdot P}\in\mathrm{Div}\,\Gamma$.

\begin{defi}
The quotient group of $\mathrm{Div}\,\Gamma$ divided by the equivalent relation $D_1\sim D_2\iff D_1-D_2=(f),\ (f\mbox{ is a rational function})$ is called the Picard group of $\Gamma$. Denote the Picard group of $\Gamma$ by $\mathrm{Pic}\,\Gamma$.
\end{defi}

A degree of a divisor $D=\sum{n_p\cdot p}$ is the integer $\sum{n_p}$. By the following theorem, we can define the degree of an element of  $\mathrm{Pic}\,\Gamma$ as well.

\begin{thm}
Let $D\in\mathrm{Div}\,\Gamma$. Then, $D\sim 0$ $\iff$ the degree of $D$ is $0$D
\end{thm}

\subsubsection{Period matrix, Jacobi variety}

Let $\gamma:[a,b]\to \Gamma$ be a continuous map. Define 
\[\textstyle
\kakko{\gamma}:=\lim_{N\to \infty}\sum_{n=0}^{N-1}{\mathrm{dist}(\gamma(\frac{nb+(N-n)a}{N}),\gamma(\frac{(n+1)b+(N-n-1)a}{N}))}.
\]

For paths $\gamma_1:[\alpha_1,\beta_1]\to\Gamma$, $\gamma_2:[\alpha_2,\beta_2]\to\Gamma$, we define a real number $(\gamma_1,\gamma_2)$ as follows: (i) Let $U_1:=[\alpha_1,\beta_1]$, $U_2:=[\alpha_2,\beta_2]$. If $\gamma_1,\gamma_2$ are injective and $\gamma_1(U_1)\cap\gamma_2(U_2)$ is connected, put $(\gamma_1,\gamma_2):=(\pm 1)\kakko{\gamma_1(U_1)\cap\gamma_2(U_2)}$, where the signature is $+$ (\textit{resp.}$-$) if the direction of $\gamma_1^{-1}\circ\gamma_2:\RR\to \RR$ is positive (\textit{resp}.~negative). 
(ii) In general case, divide the paths as $\gamma_{1,k}:[t_{k-1},t_k]\to\Gamma$, $\gamma_{2,k}:[s_{k-1},s_k]\to\Gamma$ ($\alpha_1=t_0<t_1<\dots<t_N=\beta_1$, $\alpha_2=s_0<s_1<\dots<s_M=\beta_2$) such that $\gamma_{1,k},\gamma_{2,l}$ are injective and $\gamma_{1,k}\cap\gamma_{2,l}$ are connected, and define $(\gamma_1,\gamma_2):=\sum_{k,l}(\gamma_{1,k},\gamma_{2,l})$.

We call the first Betti number $g$ of $\Gamma$ the genus of $\Gamma$. In the sequel, we fix a $\ZZ$-basis $\beta_1,\dots,\beta_g$ of $H_1(\Gamma,\ZZ)$.

\begin{defi}\label{defA5}
The period matrix of $\Gamma$ is the $g\times g$ real symmetric matrix $B$ defined by $B:=(\beta_i,\beta_j)_{i,j}$. We call the real variety $\mathrm{Jac}\,\Gamma:=\RR^g/B\ZZ^g$ the Jacobi variety of $\Gamma$.
\end{defi}

\begin{lemma}
The matrix $B$ is non-degenerate and positive definite.
\end{lemma}

Choose and fix a point $O\in \Gamma$. Let $\gamma_P$ be a path on $\Gamma$ which starts from $O$ and ends at $P$.
\begin{defi}
The Abel-Jacobi mapping starting at $O$ is the mapping $\vect{A}:\Gamma\to \mathrm{Jac}\,\Gamma;\ P\mapsto ((\beta_1,\gamma_P),(\beta_2,\gamma_P),\dots,(\beta_g,\gamma_P))\ \ (\mathrm{mod}\,B\ZZ^g)$. The value of $\vect{A}$ does not depend on the choice of $\gamma_P$. We can extend the Abel-Jacobi mapping over $\mathrm{Div}\,\Gamma$ linearly:
\[
\vect{A}:\mathrm{Div}\,\Gamma\to \mathrm{Jac}\,\Gamma;\quad \textstyle\vect{A}(\sum{n_P P})=\sum n_P\,\vect{A}(P).
\]
\end{defi}

\begin{thm}[\cite{MZ}]
The mapping $\vect{A}$ induces the homomorphism of abelian groups $\mathrm{Pic}\,\Gamma \to \mathrm{Jac}\,\Gamma$.
\end{thm}

\subsubsection{Tropical theta function and Riemann constant}\label{sec:a.1.5}

We introduce a tropical analog of the theta functions over Riemann surfaces.

\begin{defi}
The following real function $\Theta$ over $\RR^g$ is called the tropical theta function associated with $\Gamma$:
\[
\Theta(z;B):=\min_{m\in \ZZ^g}{[\frac{1}{2}\langle m,Bm \rangle+\langle m,z\rangle]},\quad
z\in \RR^g,
\]
where $\langle v,w\rangle:=\sum_{i=1}^g{v_iw_i}$ $(v,w\in\RR^g)$, and $B$ is the period matrix of $\Gamma$.
\end{defi}

Because $B$ is positive definite, the tropical theta function is well-defined over $\RR^g$. In fact, $\Theta$ is a piecewise linear convex function over $\RR^g$.
\begin{lemma}\label{hodai}
(i) Let $r\in\ZZ^g$. Then,
\begin{equation}\label{periodic}
\Theta(z+Br;B)=(-\frac{1}{2}\langle r,Br \rangle-\langle z,r \rangle)+\Theta(z;B),
\end{equation}

(ii) 
\[
\Theta(-z;B)=\Theta(z;B).
\]
\end{lemma}
\begin{proof}
It is straightforward by definition.
\end{proof}

Now we note the behavior of $\Theta$ around $\frac{1}{2}Be_i\in\RR^g$.
\begin{lemma}\label{half}
Let $\gamma$ be the map $\gamma:\RR\to\RR^g;$ $t\mapsto \frac{1}{2}Be_i+te_i$. Then the function $f(t)=\Theta(\gamma(t);B)$ satisfies $\ord_0(f)=1$.
\end{lemma}
\begin{proof}
By lemma \ref{hodai} (i), we have $f(t)=\Theta(\frac{1}{2}Be_i+te_i;B)=\Theta(-\frac{1}{2}Be_i+te_i+Be_i;B)=-t+\Theta(-\frac{1}{2}Be_i+te_i;B)$. Therefore, from lemma \ref{hodai} (ii), it follows that $\Theta(\frac{1}{2}Be_i+te_i;B)-\Theta(\frac{1}{2}Be_i-te_i;B)=-t$. Then, $f(t)-f(-t)=-t$. Because $f$ is piecewise linear, we obtain $\ord_0(f)=1$ by taking the limit $t\to 0^+$.
\end{proof}

We construct a new function over $\Gamma$ by using the Abel-Jacobi mapping and the tropical theta function. Define the multi-valued function $\tilde{\vect{A}}:\Gamma\to\RR^g$ by $\tilde{\vect{A}}:P\mapsto ((\beta_1,\gamma_P),\dots,(\beta_g,\gamma_P))$, where $\gamma_P$ is a path over $\Gamma$ from $O$ to $P$. By definition, the $\tilde{\vect{A}}$ is a lift of $\vect{A}$. Consider the multi-valued function $f:P\mapsto \Theta(\tilde{\vect{A}}(P);B)$ over $\Gamma$.
\begin{lemma}\label{zero}
The degree $\ord_P(f)$ does not depend on the choice of the branch of $f:\Gamma\to\RR$.
\end{lemma}
\begin{proof}
For a fixed point $P\in\Gamma$, let $z,z'\in\RR^g$ be two values of the multi-valued function $\tilde{\vect{A}}$ at $P$. Then, there exists some $r\in\ZZ^g$ such that $z'=z+Br$. From (\ref{periodic}), it is sufficient to prove 
\[
\ord_P(-\frac{1}{2}\langle r,Br \rangle\mathrm-\langle z,r \rangle)=0,\quad z=\tilde{\vect{A}}(P)
\]
for any $r$. Let $e_i\in\RR^g$ be the $i$-th fundamental vector. Due to the equation $\langle z,e_i\rangle=(\beta_i,\gamma_P)$ ($\gamma_P$ is a path from $O$ to $P$), the problem boils down to prove $$\ord_P((\beta_i,\gamma_P))=0.$$

Let $F$ be the multi-valued function $Q\mapsto (\beta_i,\gamma_Q)$. Take a small neighborhood $V$ of $P$. By retaking smaller $V$ if needed, we can assume  $V=\bigcup_{\alpha}{\gamma_\alpha}$ and $T_P=\{D_{\gamma_\alpha}\}$. Because $\beta_i$ is a closed path, there exist $2n$ indexes $\alpha'_1,\dots,\alpha'_n,\alpha''_1,\dots,\alpha''_n$ such that (i) $\beta_i$ comes to $P$ along to $\gamma_{\alpha'_k}$ (ii) $\beta_i$ goes from $P$ along to $\gamma_{\alpha''_k}$. By definition of $F$, we have
\begin{align*}
D_{\gamma_{\alpha'_i}}F=-1,\quad D_{\gamma_{\alpha''_i}}F=1.
\end{align*}
Therefore, $\ord_P(F)=\sum_{k=1}^n{(-1)}+\sum_{k=1}^n{(+1)}=0$.
\end{proof}

By this lemma, we can define the degree of multi-valued function $f$. We call a point $P$ which satisfies $\ord_P(f)>0$ a zero of $f$.

\begin{lemma}[\cite{MZ}]\label{ZERO}
The number of zeros of $f$ is $g=\mathrm{genus}\,\Gamma$.
\end{lemma}

\begin{defi}
Let $Q_1,Q_2,\dots,Q_g$ be the zeros of $f$ with multiplicity. The Riemann constant $\kappa$ is the element of $\mathrm{Jac}\,\Gamma$ defined by $\kappa:=\vect{A}(Q_1+\dots+Q_g)$.
\end{defi}


\begin{thebibliography}{99}

\bibitem{HHIKTT} Hatayama G, Hikami K, Inoue R, Kuniba A, Takagi T and Tokihiro T,
The $A^{(1)}_M$ automata related to crystals of symmetric tensors,
{\it J. Math. Phys.} {\bf 42} (2001) 274--308. 

\bibitem{Inoue2} Inoue R and Iwao S,
Tropical curve theory and integrable piecewise linear map,
arXiv : 11115771

\bibitem{IMS} Itenberg S, Mikhalkin G and Shustin E,
Tropical Algebraic Geometry,
(Oberwolfach Seminars Volume 35) 2009 (Birkh\"{a}user; Berlin)

\bibitem{Mada3} Iwao S, Mada J, Idzumi M and Tokihiro T,
Solution to the initial value problem of the ultradiscrete periodic Toda equation.
\textit{J.~Phys.~A:~Math.~Theor.} {\bf 42} (2009), 42315209.

\bibitem{Iwao} Iwao S, 
Solution of the generalized periodic discrete Toda equation II: theta function solution 
\textit{J.~Phys.~A: Math.~Theor.},
\textbf{43} (2010), 155208.

\bibitem{Iwao2} Iwao S,
Doctor thesis,
Tokyo university (2010)

\bibitem{Iwao3} Iwao S,
Two dimensional periodic box-ball system and its fundamental cycle,
arXiv : 11024392

\bibitem{K1} Krichever I M,
Algebro-geometric construction of the Zaharov-Shabat equations and their periodic solutions. 
\textit{Soviet Math.~Dokl.} \textbf{17} (1976), 394--397

\bibitem{K2} Krichever I M, 
Integration of nonlinear equations by the methods of nonlinear geometry.
\textit{Funk.~Anal.} \textbf{11} (1977), 15--31.

\bibitem{K3} Krichever I M, 
Methods of algebraic geometry in the theory of nonlinear equations.
\textit{Russian Math.~Surveys} \textbf{32} (6) (1977), 198--220.

\bibitem{KTT} Kuniba A, Takagi T and Takenouchi A,
  Bethe ansatz and inverse scattering transform in a periodic box-ball system, 
{\it Nucl. Phys. B} {\bf 747} (2006) 354--397.

\bibitem{Mada} Mada J, Idzumi M and Tokihiro T,
On the initial value problem of a periodic box-ball system,
\textit{J.~Phys.~A:~Math.~Gen.} {\bf 39} (2006), L617.

\bibitem{Mada2} Mada J, Idzumi M and Tokihiro T,
The box-ball system and the N-soliton solution of the ultradiscrete KdV equation,
\textit{J.~Phys.~A:~Math.~Theor.} {\bf 41} (2008), 175207.

\bibitem{MZ} Mikhalkin G and Zarkov I,
Tropical curve, their Jacobians and $\Theta$-functions,
arXiv : 0612267

\bibitem{Mumford} Mumford D, 
\textit{Tata lectures on theta II}, (Birkh\"{a}user, Boston) 1993.

\bibitem{Mumford2}
Moerbeke P and Mumford D, 
The spectrum of difference operators and algebraic curves,
\textit{Acta.~Math.} \textbf{143} (1) (1979), 93--154.

\bibitem{MINS} Murata M, Isojima S, Nobe A and Satsuma J,
Exact solutions for discrete and ultradiscrete modified KdV equations and their relation to box-ball systems,
{\it J. Phys. A: Math. Gen.} {\bf 39} (2006) L27--L34.

\bibitem{murata} Murata M, 
Lax forms of the q-Painlev\'{e} equations, {\it J. Phys. A: Math. Theor.} {\bf 42} (2009) 115201.

\bibitem{Tak04} Takagi T,
Inverse scattering method for a soliton cellular automaton,
{\it Nucl. Phys. B} {\bf 707} [FS] (2005) 577--601.

\bibitem{TS} Takahashi D and Satsuma J,
A Soliton Cellular Automaton,
{\it J. Phys. Soc. Jpn.} {\bf 59} (1990) 3514--3519.

\bibitem{T} Takahashi D,
On a Fully Discrete Soliton System,
in {\it Nonlinear evolution equations and dynamical systems} (Proceedings NEEDS '91, Baia Verde, 1991), 245--249, (World Sci. Publ., River Edge, NJ) 1992.

\bibitem{TM} Takahashi D and Matsukidaira J,
Box and ball system with a carrier and ultradiscrete modified KdV equation,
{\it J. Phys. A: Math. Gen.} {\bf 30} (1997) L733--L739.

\bibitem{tuda} Tsuda T, A geometric approach to tau-functions of difference Painlev\'{e} equations,
{\it Lett. Math. Phys.} {\bf 85} (2008) 65--78.

\bibitem{TTMS} Tokihiro T, Takahashi D, Matsukidaira J and Satsuma J, 
From Soliton Equations to Integrable Cellular Automata through a Limiting Procedure, 
{\it Phys.\ Rev.\ Lett.} {\bf 76} (1996) 3247--3250.

\bibitem{TTM} Tokihiro T, Takahashi D and Matsukidaira J,
Box and ball system as a realization of ultradiscrete nonautonomous KP equation,
{\it J. Phys. A: Math. Gen.} {\bf 33} (2000) 607--619. 

\bibitem{YT} Yura F and Tokihiro T,
On a periodic soliton cellular automaton, 
{\it J. Phys. A: Math. Gen.} {\bf 35} (2002) 3787--3801.

\bibitem{WNSTG} Willox R, Nakata Y, Satsuma J, Ramani A and Grammaticos B,
Solving the ultradiscrete KdV equation,
{\it J. Phys. A: Math. Theor.} {\bf 43} (2010) 482003 (7pp).

\end{thebibliography}
\end{document}